\documentclass[12pt]{article}
\usepackage[a4paper,foot=30mm]{geometry}
\usepackage{amsmath,amscd,amsfonts,amssymb}
\makeatletter
\let\nc\newcommand
\let\rnc\renewcommand
\let\dmo\DeclareMathOperator
\def\bb@symb#1\f@@{\expandafter\nc\csname bb#1\endcsname{\mathbb{#1}}}
\nc\bb@symbols[1]{
  \@for\@current:=#1\do{\expandafter\bb@symb\@current\f@@}}
\def\cal@symb#1\f@@{\expandafter\nc\csname cal#1\endcsname{\mathcal{#1}}}
\nc\cal@symbols[1]{
  \@for\@current:=#1\do{\expandafter\cal@symb\@current\f@@}}
\def\frak@symb#1\f@@{\expandafter\nc\csname frak#1\endcsname{\mathfrak{#1}}}
\nc\frak@symbols[1]{
  \@for\@current:=#1\do{\expandafter\frak@symb\@current\f@@}}

\bb@symbols{A,B,C,D,E,F,G,H,I,J,K,L,M,N,O,P,Q,R,S,T,U,V,W,X,Y,Z}
\cal@symbols{A,B,C,D,E,F,G,H,I,J,K,L,M,N,O,P,Q,R,S,T,U,V,W,X,Y,Z}
\frak@symbols{A,B,C,D,E,F,G,H,I,J,K,L,M,N,O,P,Q,R,S,T,U,V,W,X,Y,Z}
\frak@symbols{a,b,c,d,e,f,g,h,i,j,k,l,m,n,o,p,q,r,s,t,u,v,w,x,y,z}

\def\mym@th@p#1\f@@{\expandafter\dmo\csname#1\endcsname{#1}}
\nc\operators[1]{
  \@for\@current:=#1\do{\expandafter\mym@th@p\@current\f@@}}

\operators{Gal,Fr,Frob,Hom,End,Fil,Aut,Cores,Spec,Proj}
\operators{gr,im,coker,dlog}

\rnc{\Im}{\mathop{\mathrm{Im}}}
\rnc{\Re}{\mathop{\mathrm{Re}}}

\nc\diam[1]{\langle#1\rangle}
\nc\abs[1]{\left|#1\right|}

\nc\lowsim{\smash{\mathrel{\hbox{\lower.5ex\hbox{$\sim$}}}}}
\nc\highsim{\smash{\mathrel{\hbox{\raise.5ex\hbox{$\sim$}}}}}
\nc\tstrut[1][6]{\rule{0pt}{#1pt}}
\nc\arrowlim[2]{\mathop{\underset
    {\scriptstyle #1}{\underset{\raisebox{0ex}[0.25ex][-0.5ex]%
        {$#2$}}{\operatorname{lim}}}}}
\nc\invlim[1][]{\arrowlim{#1}{\longleftarrow}}
\nc\dirlim[1][]{\arrowlim{#1}{\longrightarrow}}
\nc\sdirlim{\rule[-1.5ex]{0pt}{0pt}\smash[b]{\dirlim}}
\nc\sinvlim{\rule[-1.5ex]{0pt}{0pt}\smash[b]{\invlim}}
\nc\hookright[1][]{\mathop{\overset{#1}%
    {\lhook\joinrel\longrightarrow}}}
\nc\hookleft[1][]{\mathop{\overset{#1}%
    {\longleftarrow\joinrel\rhook}}}
\nc\mapright[1]%
  {\stackrel{#1}{\longrightarrow}}
\nc\isomarrow{\mapright{\lowsim}}
\nc\mapleft[1]%
  {\stackrel{#1}{\longleftarrow}}
\nc\maprightonto[1]%
  {\stackrel{#1}{\longrightarrow\mkern -35mu\longrightarrow}}
\nc\inject[1]{\hookright[#1]}
\let\surject\maprightonto

\nc\backsurject[1]%
  {\stackrel{#1}{\longleftarrow\mkern -35mu\longleftarrow}}
\nc\ratmap{\mathrel{\smash-}\!\!\!\!\mathrel{\smash-}\rightarrow}

\let\N\bbN\let\Z\bbZ\let\Q\bbQ\let\C\bbC
\nc\Qbar{\overline{\bbQ}}
\nc\Ql{\bbQ_\ell}\nc\Qp{\bbQ_p}
\nc\Qlbar{\Qbar_\ell}\nc\Qpbar{\Qbar_p}
\nc\GalQ{\Gal(\Qbar/\Q)}
\nc\GalQl{\Gal(\Qlbar/\Ql)}
\nc\GalQp{\Gal(\Qpbar/\Qp)}
\nc\defeq{\mathop{:=}}
\rnc\i{^{-1}}
\nc\dual{^\vee}
\nc\mmu{\boldsymbol{\mu}}
\nc\ul{\underline}
\nc\ds{\displaystyle}

\listparindent\parindent
\usepackage{amsthm}
  \theoremstyle{plain} 
  \newtheorem{thm}[equation]{Theorem}
  \newtheorem{lem}[equation]{Lemma}
  
  \newtheorem{prop}[equation]{Proposition}
  \newtheorem*{thm*}{Theorem}
  \newtheorem*{lem*}{Lemma}
  \newtheorem*{cor*}{Corollary}
  \newtheorem*{prop*}{Proposition}
  \theoremstyle{definition}

  \theoremstyle{remark}
  \newtheorem{rem}[equation]{Remark}
  \newtheorem*{rem*}{Remark}
  
  \nc\tqed{\tag*{\hbox to 0pt{\qedsymbol\hss}}}
  \nc\noqed{\rnc{\qed}{}}
  \numberwithin{equation}{section}
  \rnc\theequation{\thesection.\arabic{equation}}

\rnc\theenumi{\roman{enumi}}
\rnc\labelenumi{(\theenumi)}

\makeatother

\usepackage[all]{xy}
\CompileMatrices
\nc\xymat[2][]{
  \entrymodifiers={++!!<0pt,.8ex>}
  \vcenter{\xymatrix#1{#2}}}
\SelectTips{cm}{12}
\UseTips
\nc\nosp[1][2.5]{\hspace{-#1em}}
\nc\flr[1]{*[l]{#1\nosp}}
\newdir{ (}{!/-1ex/@^{(}}
\newdir{ )}{!/-1ex/@_{(}}
\nc\arr{\ar@{{}-{>>}}}   
\nc{\har}{\ar@{{ (}->}}  
\nc{\hhar}{\ar@{{ )}->}} 
\nc\dar{\ar@{{}{--}>}}   
\nc\F{\bbF}
\nc\Fqbar{\overline{\F}_q}
\nc\Fq{\F_q}
\nc\Y[1]{Y_{\left\langle#1\right\rangle}}
\nc\ch{\mathop{\it char}}
\nc\Kbar{{\bar K}}
\nc\RH{\mathrm{RH}}
\nc\Ett{\tilde{E_t}}
\nc\WMC{weight-monodromy conjecture}
\theoremstyle{definition}
\newtheorem{wmc}[equation]{Weight-Monodromy Conjecture}
\operators{Frac}
\rnc\i{^{-1}}

\title{Hypersurfaces and the Weil conjectures}
\author{A J Scholl}
\begin{document}
\maketitle
\abstract{We give a proof that the Riemann hypothesis for
  hypersurfaces over finite fields implies the result for all smooth
  proper varieties, by a deformation argument which does not use the
  theory of Lefschetz pencils or the $\ell$-adic Fourier transform.}

\section*{Introduction}

Suppose that $X\subset \bbP^{d+1}$ is a smooth geometrically
irreducible hypersurface over a finite field $\Fq$. It is well-known
\cite{Weil1,Katz} that in this case the Riemann hypothesis for the zeta function of $X$
is equivalent to the point-counting estimate:
$\abs{\#X(\bbF_{q^n})-\#\bbP^d(\bbF_{q^n})}=O(q^{nd/2})$. It would
therefore be extremely interesting to have an ``elementary'' proof of
this Diophantine estimate (see for example \cite[p.299]{Katz}), in the spirit
of Stepanov's method for curves \cite{Bom,Step}. We hasten to say that
we have no idea how to do this. What we show in this paper is a
curious, albeit entirely useless, related result (Theorem
\ref{thm:main} below): one can deduce the
Riemann hypothesis for \emph{all} proper smooth $\Fq$-schemes from the
Riemann hypothesis for smooth hypersurfaces, without using monodromy
of Lefschetz pencils or Fourier transform. It is tempting to conclude
from this that an elementary proof of the Riemann Hypothesis for
hypersurfaces is unlikely to exist.

The ingredients of the proof are: the existence of the Rapoport-Zink
vanishing cycles spectral sequence, de Jong's alterations results, and
Deligne's theorem on the local monodromy of pure sheaves on curves
\cite[\S1.8]{Weil2}. (It is worth noting that the proof of the
latter theorem, although ingenious, is very short.) Of course the
rules of the game do not permit the use of the hard Lefschetz theorem
or of any consequences of the theory of weights.

The idea of deforming to hypersurfaces comes from Ayoub's work on the
conservation conjecture for motives \cite[Lemma 5.8]{Ay}.  An outline
of the proof is as follows.  Suppose $X$ is a smooth proper
$\Fq$-scheme of dimension $d$. Then $X$ is birational to a (generally
singular) hypersurface $X'\subset \bbP^{d+1}$. Deform $X'$ to a family
of hypersurfaces $H\subset \bbP^{d+1}\times \bbA^1$ whose general
member is smooth. Suppose there is a smooth connected curve $T/\Fq$
and a finite morphism $T\to \bbA^1$ such that the basechange of $H$ to
$T$ has a semistable model, $f\colon E\to T$ say. Then almost all
closed fibres $E_t$, $t\in T$ are smooth hypersurfaces, so they
satisfy the Riemann hypothesis. Let $E_\Kbar$ be the geometric generic
fibre of $f$. Then Deligne's local monodromy theorem can be applied to
the sheaves $R^if_*\Ql$ to show that at the points of degeneration of
$f$, the monodromy and weight filtrations on $H^i(E\otimes\Kbar,\Ql)$ agree.

By construction, there will be some degenerate fibre $E_t$ which is a
normal crossings divisor, and one of whose components admits a
dominant rational map to $X$. Some piece of the cohomology of $X$
therefore contributes, via the Rapoport-Zink spectral sequence, to
$H^*(E\otimes{\Kbar},\Ql)$.  Using the equality of monodromy and
weight filtrations, a calculation using the spectral sequence then
shows that the piece of $H^d(X\otimes{\Fqbar},\Ql)$ which does
contribute is pure of weight $d$, and that the remaining piece comes
from varieties of lower dimension, hence is pure by induction on
dimension.

Since the existence of semi-stable models in finite characteristic is
unknown, to make a proof one has to work with alterations
\cite{deJ1,deJ2} instead; while complicating the details somewhat,
this does not add any essential further difficulty.

\emph{Acknowledgements.}  This paper was completed while the author
was visiting Fudan University, Shanghai, and I would like to thank
Qingxue Wang and the Mathematics Department for their hospitality. I
also wish to thank Luc Illusie and Qing Liu for helpful comments.

\emph{Notations and terminology:} $p$ and $\ell$ are prime numbers
with $\ell\ne p$. We write $\bbF$ for an algebraic closure of
$\bbF_p$. If $X$ is a scheme of finite type over some separably closed
field of characteristic different from $\ell$, then $H^*(X)$ denotes
its $\ell$-adic cohomology, and $H^*_c(X)$ its $\ell$-adic cohomology
with proper support. A \emph{variety} over an algebraically closed field $k$
will mean an integral separated scheme of finite type over $k$. A
morphism of varieties is an \emph{alteration} \cite{deJ1} if it is proper,
dominant and generically finite.

\section{Weights and monodromy}
\label{sec:weights}

In this section we review basic and well-known facts concerning the
Riemann hypothesis and the \WMC.

Let $X/\bbF$ be a scheme of finite type. Then $X$ comes by basechange
from a scheme $X_0/\Fq$ for some $q=p^r$, and so
$H^*(X)=H^*(X_0\times\Spec\bbF)$ carries an action of
$\Gal(\bbF/\Fq)$. Let $F_q\in\Gal(\bbF/\Fq)$ be the geometric
Frobenius, which is the inverse of the Frobenius substitution
$x\mapsto x^q$. Recall that a finite-dimensional $\Ql$-representation
$V$ of $\Gal(\bbF/\Fq)$ is \emph{mixed of weights $\ge w$}
(resp.~\emph{mixed of weights $\le w$}, \emph{pure of weight $w$}) if
for every eigenvalue $\alpha\in\Qlbar$ of $F_q$ on $V$ there exists an
integer $j\ge0$ (resp.~$j\le0$, $j=0$) such that, for every
isomorphism $\iota\colon \Qlbar \isomarrow \C$, one has
$\abs{\iota\alpha}=q^{(w+j)/2}$. (This implies in particular that
$\alpha$ is algebraic.) The Riemann hypothesis for a smooth proper
$\bbF$-scheme $X$ is the statement:

\begin{quote}
  $\RH(X)$: \emph{For all $i\in[0,2\dim X]$, $H^i(X)$ is pure of
  weight $i$.}
\end{quote}
Since a representation of $\Gal(\bbF/\Fq)$ is
pure (or mixed) if and only if for some $r\ge1$ its restriction to
$\Gal(\bbF/\bbF_{q^r})$ is (with the same weights), this depends only
on $X$, not on the model $X_0/\bbF_q$.
For any $d\in\N$ let $\RH(d)$ be the statement
\begin{quote}
  $\RH(d)$: \emph{For all proper smooth $X$ of dimension $\le d$,
  $\RH(X)$ holds.}
\end{quote}
The following reductions are completely standard arguments, but we
include the proofs just to be clear we are not smuggling in forbidden
ingredients.

\begin{lem}\label{lem:dev1}
  Suppose $\RH(X)$ holds for all smooth projective $\bbF$-varieties
  $X$ of dimension $\le d$. Then $\RH(d)$ holds.
\end{lem}

\begin{proof}
  Suppose $X$ is a smooth proper $\bbF$-scheme of dimension $\le
  d$. We may assume $X$ is irreducible. By \cite[Thm.~4.1]{deJ1} there
  exists an alteration $f\colon X'\to X$ where $X'$ is a smooth and
  projective variety, inducing an injection $f^*\colon H^*(X)
  \inject{} H^*(X')$.
\end{proof}

\begin{lem}\label{lem:dev2}
  Suppose $X$ is smooth and projective of dimension $d$, and that
  $\RH(d-1)$ holds. 
  \begin{enumerate}
  \item[(i)] If $H^d(X)$ is pure of weight $d$ then $\RH(X)$ holds.
  \item[(ii)] If $X'$ is a variety birationally equivalent to $X$, then
    $\RH(X)$ holds if and only if, for every $i$, $H^i_c(X')$ is mixed
    of weights $\le i$.
  \end{enumerate}
\end{lem}

\begin{proof}
  (i) Let $Y\subset X$ be a smooth hyperplane section. Then if $i<d$, by weak
  Lefschetz $H^i(X)$ injects into $H^i(Y)$, which is pure of weight
  $i$ by hypothesis. Poincar\'e duality then
  gives purity for $H^i(X)$ if $i> d$. 

  (ii) There exist nonempty open subschemes $U\subset X$,
  $U'\subset X'$ with $U\simeq U'$. Let $Y=X\setminus U$,
  $Y'=X'\setminus U'$. Consider the long exact sequences of cohomology
  with proper support
  \[
  \begin{CD}
    H^{i-1}(Y) @>>> H^i_c(U) @>>> H^i(X) @>>> H^i(Y)
    \\
    @. @V\simeq VV \\
    H^{i-1}_c(Y') @>>> H^i_c(U') @>>> H^i_c(X') @>>> H^i_c(Y')
  \end{CD}
  \]
  By induction on dimension and $\RH(d-1)$, we may assume that
  $H^i_c(Y')$ and $H^i(Y)$ are mixed of weights $\le i$ for every
  $i$. Therefore 
  \[
    \text{$H^i_c(X')$ is mixed of weights $\le i$ $\iff$ $H^i(X)$ is mixed of
    weights $\le i$.}
  \]
  By Poincar\'e duality this is equivalent to the condition:
  $H^{2d-i}(X)$ is mixed of weights $\ge 2d-i$. Combining these for all
  $i$ gives the result.
\end{proof}

\begin{lem}
  \label{lem:dev3}
  Let $X$, $X'$, be smooth projective varieties, and let $f\colon
  X'\to X$ be an alteration. Let $\Gamma\subset\Aut(X')$ be a finite
  subgroup such that $f$ factors as $X'\to X'/\Gamma \mapright{g} X$,
  where $g$ is a radicial alteration.  Assume $\RH(X)$ and
  $\RH(d-1)$. Then for every $i$, the invariant subspace
  $H^i(X')^\Gamma$ is pure of weight $i$.
\end{lem}

\begin{proof}
  There exists a nonempty open $\Gamma$-invariant subscheme $U'\subset
  X'$ such that $U'\to U'/\Gamma$ is \'etale and $g\colon U'/\Gamma
  \to U=f(U')$ is finite, flat and radicial. Then $H^i_c(U) 
  \simeq H^i_c(U')^\Gamma$, and $H^i_c(U)$ is mixed of weights $\le i$ by
  Lemma \ref{lem:dev2}(ii). The obvious equivariant generalisation of
  \ref{lem:dev2}(ii) then implies that $H^i(X')^\Gamma$ is pure of weight $i$.
\end{proof}

\begin{lem}\label{lem:dev4}
  Let $X$, $X'$ be smooth projective varieties of dimension $d$, let
  $\Gamma$ be a finite group acting on $X'$, and let $f\colon
  X'/\Gamma \ratmap X$ be a dominant rational map. Assume $\RH(d-1)$
  holds. Then if $H^d(X')^\Gamma$ is pure of weight $d$, $\RH(X)$
  holds.
\end{lem}

\begin{proof}
  One can find a $\Gamma$-invariant nonempty open subscheme $U'\subset
  X'$ and an open subscheme $U\subset X$ such that $f$ induces a
  finite flat morphism $U'\to U$, and then $f^*\colon
  H^*_c(U)\inject{} H^*_c(U')^\Gamma$. Then one can
  reverse the argument of the previous lemma.
\end{proof}

Now let $R$ be a Henselian discrete valuation ring, with finite
residue field of characteristic $p$, and field of fractions
$K$. Let $G_K$ denote the absolute Galois group of $K$ and $I_K$ its
inertia subgroup. Let $t_\ell\colon I_K \to \Z_\ell(1)$ be the
tame Kummer character, given by
\[
t_\ell(\sigma)=\Bigl( \sqrt[\ell^n]{\pi}^{\,\sigma-1} \Bigr)_{n\in\N}
\]
for any uniformiser $\pi$ of $R$.

Let $\rho\colon G_K \to \Aut(V)$ be a continuous finite-dimensional
$\Ql$-representation. Then one knows that there exists a finite
extension $K'/K$ such that $\rho(I_{K'})$ is unipotent. Moreover the
restriction of $\rho$ to $I_{K'}$ factors as $\chi \circ t_\ell$ for
some homomorphism $\chi\colon \Z_\ell(1) \to GL(V)$. One defines the
\emph{logarithm of monodromy} to be $N=\log(\chi)\colon V \to V(-1)$,
which can be regarded as a nilpotent endomorphism of $V$, with twisted
Galois action. There is then an associated \emph{monodromy filtration}
$(M_n)_{n\in\Z}$ on $V$, exhaustive and separated, which is uniquely
characterised by the properties:
\begin{itemize}
\item $N(M_n(V))\subset M_{n-2}(V)(-1)$
\item For every $j\in\N$, $N^j$ induces an isomorphism $\gr^M_j(V)
  \isomarrow \gr^M_{-j}(V)(-j)$.
\end{itemize}
The monodromy filtration is stable under $G_K$, and the action of
$I_{K'}$ on $\gr^M_*(V)$ is trivial.  Write $q$ for the order of the
residue field of $K'$, so that $\Gal(\bbF/\Fq)$ acts on
$\gr^M_*(V)$. Say that $V$ is \emph{monodromy-pure of weight $w\in\Z$}
if for every $n\in\Z$, $\gr^M_n(V)$ is pure of weight $n+w$.
Deligne's \WMC\ \cite{Hodge1} is then the statement:

\begin{wmc}\label{conj:wmc}
  Let $X/K$ be smooth and proper. Then for every $i$,
  $H^i(X\otimes \Kbar)$ is monodromy-pure of weight $i$.
\end{wmc}

\begin{lem}\label{lem:except0}
  Let $V$ be monodromy-pure of weight $w$, and let $G_\bullet$ be a
  filtration on $V$ by $G_K$-invariant subspaces such that
  \begin{enumerate}
  \item[(a)] for all $n\in\Z$, $N(G_nV) \subset (G_{n-2}V)(-1)$;
  \item[(b)] for every $n\ne 0$, $\gr_n^GV$ is pure of weight $w+n$. 
  \end{enumerate}
  Then $G_\bullet=M_\bullet$.
\end{lem}

\begin{proof}
  Note that (a) implies that $I_K$ acts on $\gr_\bullet^GV$ by a
  finite quotient, so (b) makes sense. By assumption, $\gr^M_nV$ is
  pure of weight $w+n$. Assumption (b) therefore implies 
  that on $G_{-1}V$ the filtrations induced by $G_\bullet$ and
  $M_\bullet$ are equal, so that
  \begin{equation}\label{eq:GV1}
    n<0\implies \gr_n^GV=\gr_n^GG_{-1}V=\gr_n^MG_{-1}V\subset \gr_n^MV
  \end{equation}
  Dually, on $V/G_0V$ the filtrations induced by $G_\bullet$ and
  $M_\bullet$ are equal, and
  \begin{equation} \label{eq:GV2}
    n>0\implies \gr_n^MV\surject{}\ \gr_n^MV/G_0V=\gr_n^GV/G_0V= \gr_n^GV.
  \end{equation}
  Then for every $j>0$ we have a commutative diagram
  \[
  \xymat{
    \gr_j^G(V)\ar[d]_{N^j} & \gr_j^M(V)\ar[d]^{N^j}_\simeq\arr[l]
    \\
    \gr_{-j}^G(V)(-j)\har[r] & \gr_{-j}^M(V)(-j)
    }
  \]
  and therefore $N^j\colon \gr_j^G(V)\isomarrow
  \gr_{-j}^G(V)(-j)$. By the uniqueness of the monodromy filtration this
  implies $G_\bullet=M_\bullet$.
\end{proof}

We finally recall Deligne's theorem on the monodromy of pure sheaves
on curves.  Let $T/\F_q$ be a smooth curve, $U\subset T$ the
complement of a closed point $t\in T$. Let $\tilde\calO_{T,t}$ be
the henselised local ring, $K$ its field of fractions, and
$\bar\eta\colon \Spec \Kbar \to U$ the associated geometric point.

\begin{thm} {\cite[Thm 1.8.4]{Weil2}}
  \label{thm:localmono}
  Let $\calF$ be a lisse $\Ql$-sheaf on $U$ which is punctually pure
  of weight $w$ (i.e.~for every closed point $s\in U$, $\calF_{\bar
    s}$ is pure of weight $w$). Then the representation
  $\calF_{\bar\eta}$ of $\Gal(\Kbar/K)$ is monodromy-pure of weight
  $w$.
\end{thm}

\section{Vanishing cycles}

We first review Rapoport-Zink's weight spectral sequence \cite{RZ},
which is the $\ell$-adic version of Steenbrink's spectral sequence in
Hodge theory \cite{Stb}. A detailed account of all of this theory can
be found in \cite{Ill}. 

Let $S$ be a regular scheme of dimension 1, and $f\colon X\to S$ a
proper morphism. Recall that $f$ is \emph{strictly semistable} if $f$
is flat and generically smooth, $X$ is regular, and for each closed
point $s\in S$ the fibre $X_s$ is a reduced divisor with strict normal
crossings (i.e.,~the irreducible components of $X_s$ are smooth over
$s$ and intersect transversally).

Let $R$ be a Henselian DVR, and $K$ its field of fractions.  Assume
that its residue field $k$ 
has characteristic different from $\ell$. Let $\calY\to \Spec R$ be
proper and strictly semistable, with $\calY$ integral. Let $d$ be the
relative dimension of $f$, and write $Y=\calY\otimes \bar k$ for the
geometric special fibre. Thus $Y=\bigcup_{0\le i\le N}Y_i$ where each
$Y_i$ is smooth and proper, of dimension $d$. As usual we write
\begin{align*}
  Y_I=\bigcap_{i\in I} Y_i,&\qquad\emptyset\ne I\subset
  \{0,\dots,N\}
  \\
  \Y{m}=\bigsqcup_{\abs{I}=m+1}Y_I,&\qquad 0\le m\le d
\end{align*}
(our numbering differs from that of Rapoport and Zink, for whom this
would be $\Y{m+1}$). The scheme $\Y{m}$, if nonempty, is proper and
smooth of dimension $d-m$. For $a,b,n\in\Z$ set
\[
{}^nC^{a,b}=
\begin{cases}
  H^{n+2b}(\Y{a-b})(b) &\text{ if $a\ge 0\ge b$}
  \\
  0& \text{otherwise}  
\end{cases}
\]
Let
\[
\rho=\rho^{a,b,n}\colon {}^nC^{a,b} \to {}^nC^{a+1,b}
\]
be the alternating sum of the restriction maps, and
\[
\gamma=\gamma^{a,b,n}\colon {}^nC^{a,b} \to {}^nC^{a,b+1}
\]
the alternating sum of the Gysin maps.  Then
these maps make ${}^nC^{a,b}$ into a cohomological double complex; let
$({}^mC^n,d_C)$ be the associated simple complex. 
The Rapoport-Zink spectral sequence is a spectral sequence
\[
E_1^{mn}={}^mC^n \implies H^*(\calY_{\Kbar}).
\]
There is a mapping $N\colon E\to E(-1)$ on the entire
spectral sequence, given in degree $1$ by the tautological map (the
identity map when both source and target are nonzero)
\[
N\colon {}^nC^{a,b} \to {}^{n-2}C^{a+1,b+1}(-1)
\]
and which on the abutment is the logarithm of monodromy operator. In
particular, the abutment filtration\footnote{%
  The increasing filtration, normalised so that
  $\gr_nE^k_\infty=E^{k-n,n}_\infty$.}
$G_\bullet$ on $E_\infty^*$ satisfies
$N(G_m)\subset G_{m-2}(-1)$. 

Poincar\'e duality induces isomorphisms $(E_1^{m,n})^\vee \simeq
E_1^{-m,2d-n}(d)$ and (up to sign) $d_1^{-m,2d-n}$ is the transpose of
$d_1^{m-r,n+r-1}$, hence also $(E_2^{m,n})^\vee \simeq
E_2^{-m,2d-n}(d)$. (In fact the whole spectral sequence is compatible
with Poincar\'e duality on the generic and special fibres, but we will
not use this deeper fact --- see for example \cite{Sai}.)

For the rest of this section we assume that $k=\bbF$.  In this case,
the Weil conjectures for the varieties $\Y{m}$ imply that the spectral
sequence degenerates at $E_2$, and that its abutment filtration is the
weight filtration. The \WMC\ for $H^i(\calY_K)$ is then equivalent to
the statement: for every $j\ge0$ the map $N^j\colon E_\infty^{-j,i+j}
\to E_\infty^{j,i-j}(-j)$ is an isomorphism. Conversely, from the
\WMC\ one can deduce parts of the Weil
conjectures:

\begin{prop}
  \label{prop:main}
  Suppose that $\calY$ is \emph{projective} over $\Spec R$, and let
  $\calZ\subset\calY$ be a hyperplane section in general
  position. Assume:
  \begin{itemize}
  \item[(a)] $\RH(d-1)$ holds;
  \item[(b)] $\calY_K$ and $\calZ_K$ satisfy the \WMC.
  \end{itemize}
  Then the Rapoport-Zink spectral sequence degenerates at $E_2$, and
  $\RH(Y_i)$ holds for each component $Y_i$ of $Y$.
\end{prop}

(By ``hyperplane in general position'', we mean one that meets each
$\Y{m}$ transversally.)

\begin{proof}
  Assumption (a) and the proof of Lemma \ref{lem:dev2}(i) imply that if
  $(m,i)\ne(0,d)$ then $H^i(\Y{m})$ is pure of weight $i$. Therefore
  $E_1^{m,n}$ is pure of weight $n$ except possibly for $(m,n)=(0,d)$,
  and so for every $i\ne d$ the abutment filtration on $E_\infty^i$
  equals the monodromy filtration shifted by $i$.  Moreover, for
  $r\ge2$ the only differentials which can possibly be non-zero are
  those with source or target $E_r^{0,d}$, namely
  \begin{equation}
    \label{eq:diffs}
    \begin{CD}
      E_2^{-r,d+r-1}=E_r^{-r,d+r-1}
      @>{d_r^{-r,d+r-1}}>>
      E_r^{0,d} @>{d_r^{0,d}}>>
      E_r^{r,d-r+1}&=E_2^{r,d-r+1}
    \end{CD}
  \end{equation}
   We have
  $E_2^{r,d-r-1}=E_\infty^{r,d-r-1}$, since this group is not a source
  or target of any of the differentials \eqref{eq:diffs}. Also, since
  the abutment and weight filtrations on $E_\infty^{d-1}$ are equal,
  by hypothesis (b) the map
  \[
  N^r\colon E_\infty^{-r,d+r-1} \to E_\infty^{r,d-r-1}(-r).
  \]
  is an isomorphism. So
  \[
  \dim E_2^{r,d-r-1}=\dim E_\infty^{r,d-r-1}
  =\dim E_\infty^{-r,d+r-1}\le \dim E_2^{-r,d+r-1}
  \]
  with equality if and only if $d_r^{-r,d+r-1}=0$. By
  \ref{lem:critical-mono}(i) below we have equality, hence for each
  $r\ge2$, the differentials $d_r^{-r,d+r-1}$ vanish. The dual
  argument using \ref{lem:critical-mono}(ii) shows that the
  differentials $d_r^{d,0}$ also vanish. Therefore the spectral
  sequence degenerates at $E_2$.

  For the second assertion, it is enough to show that $H^d(\Y0)$ is pure of
  weight $d$.  Consider $V=H^d(\calY_\Kbar)$.  By hypothesis $V$
  satisfies the \WMC, hence $N^j\colon gr_j^W(V)\isomarrow
  gr_{-j}^W(V)(-j)$ for every $j\ge0$. The abutment filtration
  $G_\bullet$ on $V$ then satisfies the hypotheses of Lemma
  \ref{lem:except0} with $w=d$, and so $\gr^G_0V$ is also pure of weight
  $d$. Therefore $\gr_0^G(V)=E_2^{0,d}$ is pure of weight $d$. Then
  since $E_2^{0,d}$ is the middle homology of the complex
  \[
  \begin{split}
  \bigoplus_{a\ge0}H^{d-2a-2}(\Y{2a+1})(-a-1)
  \to
  \bigoplus_{a\ge0}H^{d-2a}(\Y{2a})(-a)
  \\\to
  \bigoplus_{a\ge1}H^{d-2a+2}(\Y{2a-1})(-a+1)
  \end{split}
  \]
  by hypothesis (a) we see that $H^d(\Y0)$ is also pure
  of weight $d$. 
\end{proof}

\begin{lem}
  \label{lem:critical-mono}
  (i) For every $m>0$ the map
  \[
  N^m\colon E_2^{-m,d+m-1} \to E_2^{m,d-m-1}(-m)
  \]
  is injective.

  (ii) For every $m>0$  the map
  \[
  N^m\colon E_2^{-m,d+m+1} \to E_2^{m,d-m+1}(-m).
  \]
  is surjective.
\end{lem}

\begin{proof}
  The second assertion follows from the first by Poincar\'e duality.
  Let $({}_ZE_r^{m,n})$ be the Rapoport-Zink spectral sequence
  associated to $\calZ$. It degenerates at $E_2$, since we are
  assuming the Weil conjectures in dimension $<d$. In the commutative
  square:
  \[
  \begin{CD}
    E_2^{-m,d-1+m} @>{N^m}>> E_2^{m,d-1-m}
    \\
    @V{\alpha}VV @VVV
    \\
    {}_ZE_2^{-m,d-1+m} @>{\lowsim}>{N^m}> {}_ZE_2^{m,d-1-m}
  \end{CD}
  \]
  the bottom map is an isomorphism by hypothesis (b). It is therefore
  enough to show that the map $\alpha$ is an injection. But $\alpha$ is
  induced by the vertical maps $\beta_j$ in the diagram
  \[
  \begin{CD}
    E_1^{-m-1,d-1+m} @>>> E_1^{-m,d-1+m} @>>> E_1^{-m+1,d-1+m}
    \\
    @V{\beta_{-1}}VV @V{\beta_{0}}VV @V{\beta_{1}}VV
    \\
    {}_ZE_1^{-m-1,d-1+m} @>>> {}_ZE_1^{-m,d-1+m} @>>> {}_ZE_1^{-m+1,d-1+m}
  \end{CD}
  \]
  We have for $j=0$ or $-1$
  \[
  E_1^{-m+j,d-1+m}=\bigoplus_{a\ge0} H^{d-m-2a+2j-1}(\Y{2a+m-j})(-a-m+j)
  \]
  and $\dim \Y{2a+m-j} = d-m-2a+j$, so by weak Lefschetz $\beta_j$ is
  an isomorphism for $j=-1$ and an injection for $j=0$. Therefore
  $\alpha$ is an injection.
\end{proof}

We need a variant of Proposition \ref{prop:main}, in which we are
given a finite subgroup $\Gamma\subset \Aut(\calY/R)$. Then $\Gamma$
acts on the Rapoport-Zink spectral sequence, and going
through all the steps of the previous proof, one obtains:

\begin{prop}
  \label{prop:main-G}
  Let $\calY/R$ be projective and strictly semistable, and $\Gamma\subset
  \Aut(\calY/R)$ a finite subgroup. Let $\calZ\subset \calY$ be a
  hyperplane section in general position. Assume the following hold:
  \begin{enumerate}
  \item[(a)] $\RH(d-1)$ holds.
  \item[(b)] $H^*(\calY_{\Kbar})^\Gamma$ and $H^*(\calZ_{\Kbar})$ satisfy the \WMC.
  \end{enumerate}
  Then the $\Gamma$-invariant part of the Rapoport-Zink spectral sequence for
  $H^*(\calY_{\Kbar})$ degenerates at $E_2$, and for every $i$,
  $H^i(\Y0)^\Gamma$ is pure of weight $i$.\qed
\end{prop}

\section{Deformation to hypersurfaces}
\label{sec:ayoub}

In this section we work over an algebraically closed field $k$
(ultimately $k=\bbF$), and all morphisms will be $k$-morphisms.

\begin{lem}
  \label{lem:ayoub-alter}
  Let $X$ be a smooth variety of dimension $d$.
  Then there exists:
  \begin{itemize}
  \item a projective strictly semistable morphism $g\colon E \to T$,
    where $T$ is a smooth curve;
  \item a finite subgroup $\Gamma\subset \Aut(E/T)$;
  \item a nonempty open subscheme $U\subset T$ and a family of smooth
    hypersurfaces $Z\subset \bbP^{d+1}\times U$, together with a
    dominant $U$-morphism $E\times_TU \to Z$, inducing a purely inseparable
    inclusion of function fields $\kappa(Z) \subset\kappa(E)^\Gamma$;
  \item a point $t\in T(k)$ and a dominant rational map from
    $E_t/\Gamma$ to $X$.
  \end{itemize}
\end{lem}

\begin{proof}
  There exists a birational morphism $p\colon X\to H_0$ where
  $H_0\subset \bbP^n$ is an integral hypersurface (in general
  singular). Choose a pencil $f\colon H \to \bbP^1$ of hypersurfaces
  whose generic fibre is smooth, and with $f\i(0)=H_0$.  We then apply
  the procedure of the last paragraph of the proof of
  \cite[5.13]{deJ2} to $f$, taking the group $G$ in
  \emph{loc.~cit.}~to be trivial. This gives a commutative diagram of
  varieties and dominant projective morphisms
  \[
  \begin{CD}
    E @>g>> T\\
    @VVV @VV\pi V \\
    H @>>f> \bbP^1
  \end{CD}
  \]
  in which $T$ is a smooth curve and $E$ is regular. The unlabelled
  vertical morphism is an alteration, and $g$ (which is obtained by
  repeated blowups from a ``$G'$-pluri nodal fibration'', in the
  terminology of \emph{loc.~cit.}) is strictly semistable. Finally
  there is a finite group $\Gamma$ (which is $N$ in de~Jong's
  notation) acting on $E$ covering the trivial actions on $H$ and $T$,
  such the extension $k(E)^\Gamma/k(H\times_{\bbP^1}T)$ is purely
  inseparable. We take $U\subset T$ sufficiently small such that
  $Z=H\times_{\bbP^1}U$ is smooth over $U$.  Finally, for any point
  $t\in\pi\i(0)$ we get a dominant rational map $E_t/\Gamma \to H_0\ratmap
  X$.
\end{proof}

\begin{rem}
  Ayoub's lemma \cite[Lemma 5.8]{Ay} is similar but stronger, since an
  alteration is not required in characteristic zero, as one may in
  that case appeal to semi-stable reduction \cite{KKMS}.
\end{rem}

We now suppose that $k=\bbF$. 

\begin{thm}\label{thm:main}
  Assume that $\RH(d-1)$ holds, and that $\RH(V)$ holds for every smooth
  hypersurface $V\subset \bbP^{d+1}_{\bbF}$. Then $\RH(d)$ holds.
\end{thm}

\begin{proof}
Let $X/\bbF$ be smooth and proper of dimension $d$. 
Choose $E\to T$ be as in the lemma. Fix a projective
embedding $E\inject{} \bbP^D\times T$, and choose a hyperplane
$L\subset \bbP^D$ which meets $E_t$ in general position. Let $W=E\cap
(L\times T)\mapright{g}T$ be the hyperplane section given by
$L$. Replacing $T$ by an open subscheme contaning $t$, and $E$ by its
inverse image by $f$, we may assume that for every $s\in U=T-\{t\}$
the fibres $H_{\pi(s)}$, $E_s$ and $W_s$ are smooth, and that the
morphism $E_s/\Gamma \to H_{\pi(s)}$ is the composition of a
modification and a radicial morphism. The sheaves
\[
\calF^i=R^if_*\Ql|_U,\qquad\calG^i=R^ig_*\Ql|_U
\]
and $(\calF^i)^G$ are therefore lisse $\Ql$-sheaves on $U$.  Let
$R=\tilde\calO_{T,t}$, $K=\Frac(R)$ and $\bar\eta\colon \Spec\Kbar
\to U$ as in \S\ref{sec:weights}.

\begin{prop}
  $H^i(E_{\bar\eta})^\Gamma$ and $H^i(W_{\bar\eta})$ are monodromy-pure of
  weight $i$.
\end{prop}

\begin{proof}
  Let $s\in U$ be a closed point.  By hypothesis $\RH(W_s)$ holds, and
  by \ref{lem:dev3} and the hypotheses,
  $H^i(E_s)^\Gamma$ is pure of weight $i$.  Thus the sheaves
  $(\calF^i)^\Gamma$ and $\calG^i$ are punctually pure of weight $i$,
  and Theorem \ref{thm:localmono} then implies that
  $H^i(E_{\bar\eta})^\Gamma=(\calF^i_{\bar\eta})^\Gamma$ and
  $H^i(W_{\bar\eta})=\calG^i_{\bar\eta}$ are monodromy-pure of weight
  $i$.
\end{proof}

Let $\Ett$ be the normalisation of $E_t$, and apply
Proposition \ref{prop:main-G} with $\calY=E\times_T\Spec R$, $\calZ=
W\times_T\Spec R$. We conclude that $H^i(\Ett)^\Gamma$ is
pure of weight $i$. Now let $X'\subset E_t$ be any component that
dominates $X$ under the rational map $E_t/\Gamma \to X$, and let
$\Delta\subset G$ be its stabiliser. Then $H^i(X')^\Delta \subset
H^i(\Ett)^\Gamma$. So $H^i(X')^\Delta$ is pure of weight
$i$, and so by Lemma \ref{lem:dev4}, the Riemann hypothesis holds for
$X$.
\end{proof}

\vspace*{0.5cm}
\noindent
Department of Pure Mathematics and Mathematical Statistics\\
Centre for Mathematical Sciences\\
Wilberforce Road\\
Cambridge\enspace CB3 0WB\\
England\\
\texttt{a.j.scholl@dpmms.cam.ac.uk}

\end{document}